\documentclass[11pt]{amsart}
\usepackage{amsfonts,amsmath,amssymb,amsthm}

\newtheorem{theorem}{Theorem}[section]
\newtheorem{lemma}[theorem]{Lemma}
\newtheorem{proposition}[theorem]{Proposition}
\newtheorem{corollary}[theorem]{Corollary}

\theoremstyle{definition}

\theoremstyle{remark}

\numberwithin{equation}{section}

\newcommand{\CC}{\mathbf C}\newcommand{\NN}{\mathbf N}

\newcommand{\RR}{{\mathbf  R}}
 
\newcommand{\RRn}{{\RR}^n}

\newcommand{\ep}{\varepsilon}

\newcommand{\hake}[1]{\langle #1 \rangle }

\newcommand{\scalar}[2]{\langle #1 ,#2 \rangle }

\newcommand{\norm}[1]{\Vert #1 \Vert }
\newcommand{\normrum}[2]{{\norm {#1}}_{#2}}

\newcommand{\upp}[1]{^{(#1)}}
\newcommand{\p}{\partial}

\newcommand{\opn}{\operatorname}

\newcommand{\Cal}{\mathcal}

\newcommand{\ee}{\text{\rm e}}
\newcommand{\ii}{\text{\rm i}}
\newcommand{\dd }{\,\text{\rm d}}
\newcommand{\bS}{\mathbf S}

\title{Analysis of the quadratic term in the backscattering transformation}

\author[Ingrid Belti\c t\u a]{Ingrid Belti\c t\u a $^*$}
\thanks{$^*$ Partially supported by 
 the grant 2-CEx06-11-18/2006}
\address{Institute of Mathematics "Simion Stoilow"
of the Romanian Academy, PO Box 1-764, RO 014700 Bucharest, Romania}
\email{Ingrid.Beltita@imar.ro} 
\author[Anders Melin]{Anders Melin}
\address{Lund University,  Sweden}
\email{andersmelin@hotmail.com}

\begin{document}

\parskip=5pt
\baselineskip=18pt

\maketitle
\begin{abstract}
The quadratic term in the Taylor expansion at the origin of the backscattering transformation in odd dimensions $n\ge 3$ gives rise to a symmetric bilinear operator $B_2$ on 
$C_0^\infty(\RR^n)\times C_0^\infty(\RR^n)$.
In this paper we prove that $B_2$ extends to certain Sobolev spaces with weights
and show that it improves both regularity and decay.
\end{abstract}

\section{Introduction and formulation of the main result}

The quadratic part obtained when the quantum backscattering data are
expanded in a power series in the potential  gives rise to a symmetric
bilinear operator $B_2$  on $C_0^\infty (\RRn)\times C_0^\infty(\RR^n)$  when $n \geq 3$ is
odd. 
We refer to the paper \cite{M:contemp}  in which the backscattering transform
was defined in arbitrary odd dimension (see Definition 3.4 in that
paper).  
An explicit formula for $B_2 $ is provided by Corollary 10.7
of \cite{M:contemp}, which implies that
\begin{equation}\label{b2formula}
B_2(f,g)(x) = \iint E(y,z) f(x+ \frac{y+z}2) g(x-\frac{y-z}2) \dd y
\,\dd z,\
\quad f, g\in C_0^\infty (\RRn).
\end{equation}
Here $E(y,z)= 4^{-1}(i{\pi})^{1-n}\delta \upp{n-2}(|y|^2-|z|^2) $ is
the unique fundamental solution  of the ultra-hyperbolic operator
$\Delta _y-\Delta _z$ such that $E(y,z)=-E(z,y)$ and $E(y,z)$ is separately
rotation invariant in both variables (see Corollary 10.2 of
\cite{M:contemp}).

Since $E$ is not a function the  formula \eqref{b2formula} needs to be
interpreted in the distribution sense. 
If the trilinear form $Q$ on
$C_0^\infty(\RRn)$ is defined through
\begin{equation}\label{Q:def}
Q(f,g,h) = \int h(x)B_2(f,g)(x)\dd x , \quad f, g, h \in C_0^\infty (\RRn),
\end{equation}
this means that
$$
Q(f,g,h) = \scalar{E}{{\Phi}}
$$
where
$$
{\Phi}(y,z)= \int h(x) f(x+ \frac{y+z}2) g(x-\frac{y-z}2)\dd x.
$$
In what follows we are  going to use similar notation for integrals
that have to be interpreted in the distribution sense.

The expressions for $B_2$ and $E$ above show easily that $B_2$ is
continuous from $C_0^\infty (\RRn)\times C_0^\infty (\RRn) $ to
$C_0^\infty (\RRn)$. 
It commutes with translations and $|x| \leq
\sqrt{r_1^2+r_2^2} $ in the support of $B_2(f,g)$ if $|x|\leq r_1$ in
$\opn{supp}(f)$ and $|x| \leq r_2$ in $\opn{supp} (g)$. 
From  formulas we derive in the next section it will be clear  that
$B_2$ extends 
to a much larger domain than  $C_0^\infty \times C_0^\infty$. 
In particular $B_2(f,g)$ is defined as a distribution  when $f$, $g \in
L_{\opn{cpt}}^2$. 

In this paper we are going to derive  continuity estimates for $B_2$
in weighted Sobolev spaces.  
Specifically,  we consider the  spaces
$$
H_{(a,b)}(\RR ^d) = \{u \in \Cal S'(\RR ^d); \hake x ^a \hake D ^b u
\in L^2 (\RR ^d) \}
$$
where $a, b \in \RR$ and $D= \ii^{-1} \partial$, hence $\hake D$ is multiplication by $\hake {\xi}
=(1+|{\xi}|^2)^{1/2}$ on the Fourier transform side.  
We shall prove then that for certain values of $a,b \geq 0$ it is true that
$B_2$ extends to a bilinear operator on $H_{(a,b)}$.  
In fact, it
happens also that there are $a$, $b$, $\bar a$, $\bar b$ with $a<\bar a$, $b <\bar b$ such that $B_2$ is continuous from $H_{(a,b)}(\RRn) \times
H_{(a,b)}(\RRn) $ to $ H_{(\bar a, \bar b)}(\RRn)$. 
This means that
$B_2$  in a certain sense improves decay and regularity at the same
time and therefore shares some nice features with ordinary
multiplication as well as convolution. 
There are good reasons to
believe that similar properties hold for higher order terms $B_N$ in
the expansion of the backscattering transform, and if so, this would
have applications in inverse scattering.

Throughout this  paper we use the notation
$m =(n-3)/2$. 
Our main result is  the following theorem.
\begin{theorem}\label{mainthm}
Assume $(a', b',a'',b'', a, b)\in \RR^6$ satisfies 
\begin{equation}\label{exp:cond}
\begin{gathered}
0<a< m+1/2+\min (a',a''), \quad a\le a'+ a'' -1/2, \\
0\le b <1 + \min  (b',b''), \quad b+m\le b'+ b'',\\
a+b < 1/2 +  \min (a',a'')+\min (b',b'').
\end{gathered}
\end{equation}
Then $B_2$ is continuous from
$ H_{(a', b')}(\RR^n)\times  H_{(a'', b'')}(\RR^n)$ to
$H_{(a, b)}(\RR^n)$.
\end{theorem}

\begin{corollary}
Assume that 
$$
0\le \bar a \leq a, \; 0\le a, \;  0 \leq \bar b \leq b, \; \bar a + \bar b < 1/2.
$$ 
Then $B_2$ is continuous from $H_{(1/2+a,m+b)}\times H_{(1/2+a,m+b)}$ to 
$H_{(1/2+a+\bar a, m+b +\bar b)}$. 
In particular, $B_2$ is a
continuous bilinear operator on $H_{(a,b)}$ when $a\geq 1/2$ 
and $b
\geq m$.
\end{corollary}

The proof of Theorem~\ref{mainthm} relies on a  duality argument applied to the
trilinear form $Q$ in \eqref{Q:def}.  
There is a simpler expression for $Q$. 
To see this, consider the bilinear operator
$$
A: \Cal S(\RR ^n) \times \Cal S(\RR ^n ) \to \Cal S (\RR ^{n+1})
$$
defined
through
\begin{equation}\label{def:A}
A(f,g)(x,t) =  \int \limits  k_0 (y, t) f(x-y)g(x+y) \,
\dd y , \quad x \in \RR ^n, \,  t \in \RR. 
\end{equation}
Here
$
k_0(x,t)$
is the convolution kernel of the operator $K_0(t) =\sin (t|D|)/ |D|$.

\begin{lemma}\label{lemma:Q}
We have the identity
\begin{equation}\label{1AM}
 Q(f, g, h) = -4 \iint\limits_{\RR^n\times \RR^+} 
A(f, g)(x, t) (\cos t|D| h)(x) \dd x\dd t
\end{equation}
when $f$, $g$, $h \in C_0^\infty(\RR^n)$.
\end{lemma} 

\begin{proof}
By using the homogeneity of $E$ we get
\begin{equation}\label{2AM}
Q(f,g,h)= 4 \iiint E(y,z)h(x-z)f(x+y)g(x-y)\dd x\dd y\dd z.
\end{equation}

We recall that (see Thm.~10.4 of \cite{M:contemp})
\begin{equation}\label{4AM}
E(y,z) = -\int\limits _0^\infty  k _0(y,t) \dot k_0(z,t) \dd t.
\end{equation}
Then we  first integrate with respect to $z$ in \eqref{2AM} and apply \eqref{4AM} to  write
\begin{equation}\label{5AM}
\int E(y,z)h(x-z) \dd z = - \int _0^\infty k_0(y,t) (\cos(t|D|)h)(x) \dd t.
\end{equation}
Then the integration with respect to $y$ in \eqref{2AM} gives
\begin{equation}\label{6AM}
\int k_0(y,t) f(x+y)g(x-y)\dd y =  A(f,g)(x,t).
\end{equation}
The formula \eqref{1AM} is then obtained by integrating over the  remaining variables $x$ and $t$. 
\end{proof}

The main idea is to use continuity properties of the operators 
$\cos (t|D|)$ and $A$ in $H_{(a, b)}$ spaces in  order to get the needed estimates.

Continuity properties of $\cos (t|D|)$ and $A$ are  obtained in section~3, and the 
proof of the main result is then derived.

\section{An interpolation result for bilinear operators}
In this subsection we consider general dimensions  $ d\ge 1$.

When $a,b\in \RR$ we define
\begin{equation}\label{hab:def}
H_{(a,b)}(\RR^d )=\{u \in \Cal S'(\RR^d);\, \hake x ^a \hake D ^b u\in
L^2 (\RR^d ) \}.
\end{equation}
This is a Hilbert space with norm
$$
\normrum u{(a,b)}= \norm{\hake x ^a \hake D ^bu}, 
$$
where the norm in the right-hand side is the $L^2$ norm.
Since the operators $\hake D ^b \hake x ^a \hake D ^{-b} \hake x
^{-a}$ and $\hake x ^a \hake D ^b \hake x ^{-a}\hake D ^{-b}$ are
continuous in $L^2$,  it follows that
 $$
 H_{(a,b)}(\RR^d)= \{u \in \Cal S'(\RR^d); \, \hake D ^b \hake x ^a u
 \in L^2 (\RR^d)\}
$$
and the norms $\normrum u{(a,b)}$ and $\normrum u{(a,b)}'=\norm{\hake
  D ^b \hake x ^a u}$ are equivalent.
This in turn implies that the Fourier transform is a linear
homeomorphism  from $H_{(a,b)}$ onto $H_{(b,a)}$.

Assume  $T \colon {\Cal S}(\RR^d) \times {\Cal S}(\RR^d) \to {\Cal S}(\RR^N)$
is a continuous bilinear  operator. 
Let $I(T)$ be the set of all $\sigma= (a', b', a'', b'', a, b) \in \RR^6$
for which there is a constant $C=C(\sigma)$ such that
\begin{equation}\label{interp:1}
\normrum{T(f, g)}{(a, b)} \le C \normrum{f}{(a', b')}\normrum{g}{(a'', b'')}, \quad
f, \, g \, \in \, {\Cal S}(\RR^d).
\end{equation}

The next theorem might be obtained as an application of Theorem~4.4.1 in \cite{BL}. 
For the reader's convenience we include here a direct proof.

\begin{theorem}\label{interp:thm}
The set $I(T)$ is convex in $\RR^6$.
\end{theorem}

We are going to use the following lemma.

\begin{lemma}\label{interp:lemma}
Assume $K \subset \RR$ is a compact set.
Then there is a positive constant $C$ depending on $K$ and $d$ only
such that 
\begin{equation}\label{interp:2}
\normrum{\hake{D}^b \hake{x}^{\ii t} \hake{D}^{-b}}{L^2(\RR^d)\to L^2(\RR^d) }
\le (1+ C|t|)^{|\opn{Re} b|},
\end{equation}
when $\opn{Re} b\in K$ and $t\in \RR$.
\end{lemma} 

\begin{proof}
Choose a positive integer $M$ such that $ K\subseteq [-2M, 2M]$ and set
$$ 
P_M(z, t) = \hake{D}^{2Mz} \hake{x}^{\ii t} \hake{D}^{-2Mz}
$$
when $z\in \CC$, $t \in \RR$. 
We have that 
$$ D_j \circ \hake{x}^{\ii t} = \hake {x}^{\ii t }\circ D_j + tx_j \hake{x}^{\ii t-2}.
$$
It follows by induction over $|\alpha|$ that 
$$ 
 D^\alpha \circ \hake{x}^{\ii t} = \hake {x}^{\ii t }\circ D^\alpha +
\sum\limits_{\stackrel{1\le k \le |\alpha|}{|\beta|< |\alpha|}}
\hake{x}^{\ii t} t^k p_{k, \alpha, \beta}(x) \circ  D^\beta,
$$ 
where $\hake{x}^{|\gamma|} \partial ^\gamma p_{k, \alpha, \beta}$ is bounded  for every $\gamma\in \NN^d$. 
Hence there is a constant $C$, which depends on $M$ and $d$ only, such 
that 
$$ \normrum{P_{M}(1, t)}{L^2 \to L^2} \le (1+ C|t|)^{2M}.
$$
Since $ \hake{D}^{2M \ii \opn{Im} z}$ is unitary in $L^2$ it follows that
\begin{equation}\label{interp:3}
\normrum{P_{M}(z, t)}{L^2 \to L^2} \le (1+ C|t|)^{2M} \quad \text{when} 
\quad \opn{Re} z =1, 
\; t\in \RR.
\end{equation}
One also clearly has that 
\begin{equation}\label{interp:4}
\normrum{P_{M}(z, t)}{L^2 \to L^2} \le 1 \quad \text{when} \quad \opn{Re} z =0, 
\; t\in \RR.
\end{equation}

Let $f$, $g\in {\Cal S}(\RR^d)$ satisfy 
$\norm{f}=\norm{g}=1$ and set
$$ q_M(z, t) = (1+C|t|)^{-2zM} \scalar{P_M(z, t) f}{g}.$$
This is an entire analytic function, bounded on the strip  $ 0\le \opn{Re} z \le 1$
and, by \eqref{interp:3} and \eqref{interp:4}, $|q_M(z, t)|\le 1$ when $\opn{Re} z= 0$ or 
$\opn{Re} z= 1$.
It follows by the three lines theorem that $|q_M(z, t)|\le 1$ when $0\le \opn{Re} z\le 1$.
This implies that
$$
\normrum{P_{M}(z, t)}{L^2 \to L^2} \le (1+C|t|)^{2M |\opn{Re} z|}$$
when $0\le \opn{Re} z\le 1$, $ 
\; t\in \RR$.
A similar proof shows that the above  inequality holds also for 
$-1\le \opn{Re} z\le 0$.
The lemma follows after replacing $z$ by $b/(2M)$.
\end{proof}

\begin{proof}[Proof of Theorem~\ref{interp:thm}]
Assume 
$$\sigma_0 = (a_0', b_0', a_0'', b_0'', a_0, b_0), \quad
\sigma_1 = (a_1', b_1', a_1'', b_1'', a_1, b_1)
$$
are elements of $I(T)$.
Define
$$ \sigma(z) = (a'(z), b'(z), a''(z), b''(z), a(z), b(z))=(1-z) \sigma_0 + z\sigma_1, \quad z\in \CC.
$$
Let $f$, $g \in {\Cal S}(\RR^d_x)$ and $h \in {\Cal S}(\RR^N_y)$ and 
set
$$ 
F(z) = \hake{x}^{-a'(z)} \hake{D}^{-b'(z)} f, \quad 
G(z) = \hake{x}^{-a''(z)} \hake{D}^{-b''(z)} g
$$
and
$$ H(z) = \hake{y}^{a(z)} \hake{D}^{b(z)} h.$$
Then $F$, $G$ and $H$ are holomorphic functions of $z$ with values in ${\Cal S}(\RR^d)$ 
and ${\Cal S}(\RR^N)$, respectively, and their ${\Cal S}$ seminorms have at most polynomial growth in $|z|$ when $\opn{Re} z$ stays in a bounded set.

The previous lemma shows that when $\opn{Re} z=0$
$$
\begin{gathered} 
 \normrum{F(z)}{(a'_0, b'_0)} \le C_1 
\norm{\hake{D}^{b'_0} \hake{x}^{z(a'_0-a'_1)}\hake{D}^{-b'_0} \hake{D}^{z(b'_0- b'_1)} f}
\\
\le C_2 (1+ |\opn{Im} z|)^{|b'_0|} \norm{\hake{D}^{z(b'_0- b'_1)} f}
\le C |1+z|^{|b'_0|} \norm{f}. 
\end{gathered}
$$
Similarly one gets
$$ 
 \normrum{G(z)}{(a''_0, b''_0)} \le C |1+z|^{|b''_0|} \norm{g}, \qquad 
 \normrum{H(z)}{(-a_0, -b_0)} \le C |1+z|^{|b_0|} \norm{h}
$$
when $\opn{Re} z=0$, and 
$$
\begin{gathered}
 \normrum{F(z)}{(a'_1, b'_1)} \le C |1+z|^{|b'_1|} \norm{f}, \qquad 
\normrum{G(z)}{(a''_1, b''_1)} \le C |1+z|^{|b''_1|} \norm{g} \; \text{ and} \\
\normrum{H(z)}{(-a_1, -b_1)} \le C |1+z|^{|b_1|} \norm{h}
\end{gathered}
$$
when $\opn{Re} z=1$.

Define
$$ q(z) = \scalar{T(F(z), G(z))}{H(z)}.
$$
This is an entire analytic function.

Since $T$ is continuous  from
$\Cal S(\RR ^d) \times \Cal S ( \RR ^d) $ to $\Cal S (\RR ^N)$ it
follows (by using commutator estimates as in  the previous lemma)
that $q(z)$ is of most polynomial growth in the strip $0 \leq 
\opn{Re} z \leq 1$. 
Since $ {\sigma}_0$, ${\sigma}_1
\in I(T)$  the estimates for $F$, $G$, $H$ above show that there are
positive constants $C$ and ${\gamma}$, which are independent of
$f$, $g$, $h$, 
such that $|(1+z)^{-{\gamma}}q(z)| \leq C \norm f \cdot \norm g \cdot
\norm h$ when $\opn{Re} 
z=0$ or $\opn{Re} z=1$. 
It follows then from the three lines theorem 
that $(1+z)^{-{\gamma}}q(z)$ satisfies the same estimate for every $z$ in  the whole strip. 
When $z = {\theta} \in (0,1)$ we  get an estimate for
$q({\theta})$, and hence the   estimate
$$ \norm{\hake{D}^{b(\theta)} \hake{y}^{a(\theta)} 
T (\hake{x}^{-a'(\theta)} \hake{D}^{-b'(\theta)}f, 
\hake{x}^{-a''(\theta)} \hake{D}^{-b''(\theta)}g)}
\le C \norm{f}\, \norm{g}, $$
where $C$ is independent of $f$ and $g$. 
This means precisely that ${\sigma}({\theta}) \in I(T)$. 
\end{proof}

\section{Proof of the main result}

We recall that $n\ge 3$  is odd and we have denoted $m=(n-3)/2$.
We define the operator $K: \Cal S '(\RR ^n) \to \Cal S'(\RR ^{n+1})$
through
\begin{equation}\label{K:def}
(Ku)(x,t) = Y_+(t) \cos (t|D|)u(x),\quad t\in \RR, \, x\in \RR^n, 
\end{equation}
where $Y_+$ is the characteristic function of $[0, \infty)$.

\begin{proposition}\label{Prop:K}
Assume $a < 0$ and $b\le 0$. 
Then the operator $K$ is continuous from
$H_{(a,b)}(\RR ^n) $ to $H_{(a-1/2,b)}(\RR ^{n+1})$.
\end{proposition}

\begin{proof}
When $t\geq 0$  we denote by $\Lambda _t$ the operator on $\Cal
S(\RRn)$ which is multiplication by the function
$((1+t)^2+|x|^2)^{1/2}$ and we  consider $K(t)=\cos (t|D|)$ as an
operator in $\Cal S (\RRn)$. 
Since $n$ is odd the convolution kernel
$\dot k_0(x,t)$ of $K(t)$ is supported in the set where
$|x|=t$. 
Therefore $K(t)f$ is supported in the ball with centre $x_0$
and radius $r+t$ if $f \in C_0^\infty (\RRn)$ is supported in the
ball with centre $x_0$ and radius $r$.

Assume first that $a$ is real, arbitrary.  
We prove that there exists
a constant $C_a$ such that
\begin{equation}\label{ein}
\normrum{\Lambda _t ^{-a}K(t) {\Lambda}_t ^a}{L^2(\RRn) \to L^2(\RRn)}
\leq C_a, \quad t \geq 0.
\end{equation}
Let $(T_{\sigma})_{{\sigma}>0}$  be the dilation group on $\Cal S
(\RRn)$ defined by $T_{\sigma}h(x)= {\sigma}^{n/2}h({\sigma}x)$. 
Then
$T_{\sigma} ^{-1} = T_{1/{\sigma}}$ and $T_{\sigma}$ extends to a
unitary operator in $L^2$ for every ${\sigma}$. 
We notice that
$$
T_{\sigma}K(t)T_{\sigma}^{-1}= K(t/{\sigma})
$$
and
$$
T_{1+t}{\Lambda}_t ^a T_{1+t}^{-1} = (1+t)^a {\Lambda}_0^a.
$$
It  follows that
$$
{\Lambda} _t ^{-a} K(t) {\Lambda}_t^a = T_{1+t}^{-1}\Lambda _0 ^{-a}
K(t/(1+t)) {\Lambda}_0^aT_{1+t}.
$$
Therefore, it is enough to show that for $0\le t\le 1$ the operator $\Lambda_0^{-a} K(t) \Lambda_0^a$ extends to a
bounded operator on $L^2(\RR^n)$ and that there exists 
$C_a>0$ such that
\begin{equation}\label{K:0}
\normrum{\Lambda_0^{-a} K(t) \Lambda_0^a}{L^2(\RR^n)\to L^2(\RR^n)}\le C_a, \qquad 0\le t\le 1.
\end{equation}

Take $0\le t\le 1$. 
We notice that $(K(t)f,K(t)g)=0$ if
$\opn{dist}(\opn{supp}(f), \opn{supp}(g))>2$, since the supports of
$K(t)f$ and $K(t)g$ do not overlap.
Let $0\leq \chi \in C_0^\infty(\RR ^n)$
be supported in the unit ball and satisfy $\int \chi (y)\dd y =1$. 
For
$f \in C_0^\infty(\RR^n) $ define $f_y(x) = f(x) \chi (x-y)$. 
Then
$$
(K(t)f_y ,K(t) f_z) = 0 \qquad 
\text{when}\quad  |y-z| \geq 4.$$
Since $ f = \int f_y \dd y$ it follows that
$$
\begin{gathered}
(\Lambda^{-a}_0 K(t)\Lambda_0^{a} f,\Lambda^{-a}_0 K(t) \Lambda_0^{a} f)= \iint 
(\Lambda^{-a}_0 K(t)\Lambda_0^{a} f_y,\Lambda^{-a}_0 K(t) \Lambda_0^{a} f_z) \dd y \dd z
\\
= \iint \limits _{|y-z|\leq 4} 
(\Lambda^{-a}_0 K(t) \Lambda_0^{a} f_y,\Lambda^{-a}_0 K(t)\Lambda_0^{a} f_z) \dd y \dd z
\leq C \int \norm{ \Lambda^{-a}_0 K(t)\Lambda_0^{a} f_y    }^2 \dd y.
\end{gathered}
$$ 
Since $|x-y|\le 1$ in the support of $\Lambda^a_0f_y$, we have that
$|x-y|\le t+1$ in the support of $K(t) \Lambda_0^a f_y$. 
Hence
$$ \norm{\Lambda_0^{-a} K(t) \Lambda^a_0 f_y}\le C_1 \hake{y}^{-a} \norm{\Lambda^a_0 f_y}
\le C \norm{f_y}.
$$
The proof of \eqref{K:0} is then completed by the fact that 
$$ \int \norm{f_y}^2 \dd y \le C\norm{f}^2.$$

Using \eqref{ein} we get, when $a<0$, 
$$
\begin{gathered}
\iint (1+|x|^2+t^2)^{a-1/2}|Ku(x,t)|^2 \dd x \dd t =
\iint (1+|x|^2+t^2)^{a-1/2}|(K(t) u)(x)|^2 \dd x \dd t
\\
\leq C \iint (1+|x|^2+t^2) ^{a-1/2}|u(x)|^2 \dd x \dd t\\
= C \Big (\int _{-\infty}^\infty (1+t^2) ^{a-1/2}\dd t \Big ) \int
(1+|x|^2)^a |u(x)|^2 \dd x = C' \normrum u{(a,0)}^2
\end{gathered}
$$
when $u \in {\Cal S}(\RR^n)$.
This concludes the proof for the case $b=0$, since $\Cal S(\RR^n)$ is 
dense in $H_{(a, 0)}$.

In the case $b<0$ the proposition follows from the
fact that $K$ commutes with $D_x$ and
the operator $(1+|D_t|^2+ |D_x|^2)^b\hake{D_x}^{-b}$ is bounded.  
\end{proof}

The previous proposition combined with Lemma~\ref{lemma:Q} gives the next corollary. 

\begin{corollary}\label{cor:K}
Assume  $a_1$, $a_2$, $b_1$, $b_2$, $a_3$, $b_3 \in \RR $, $a_3 >0$ and $b_3\ge 0$. 
Then
$B_2$ is continuous from $H_{(a_1,b_1)}\times H_{(a_2,b_2)}$ to
$H_{(a_3,b_3)}$ if $A$ is continuous from $H_{(a_1,b_1)}\times
H_{(a_2,b_2)}$ to $H_{(a_3+1/2,b_3)}$.
\end{corollary}

We turn our attention to proving continuity properties for the bilinear operator $A$.
We will first establish some useful formulas for $A(f, g)$ and its Fourier transform.

Let 
$
S\colon \Cal S(\RR ^n) \times \Cal S(\RR ^n ) \to \Cal S (\RR ^{n+1})
$
be the operator defined
through
\begin{equation}\label{def:B}
S(f,g)(x,t) =  t^{m+1} \int \limits_{\bS^{n-1}}   f(x+t \omega)g(x - t\omega) \,
\dd \omega, \quad x \in \RR ^n, \,  t \in \RR. 
\end{equation}
It is easy to see that $S$ extends to a bounded operator from 
$L^2(\RR^n)\times L^2(\RR^n)$ to $L^2(\RR^{n+1})$.

\begin{lemma}\label{lemma:fourier}
Let $\widehat A(f,g)({\xi},{\tau})$
denote the Fourier transform  of $A(f, g)$ with respect to both variables.
Then
\begin{equation}\label{fourier:1}
\widehat A(f,g)({\xi},{\tau}) =\frac {(\tau/2)^m}{2^3 i(2\pi)^{n-1}} S(\hat f, \hat g)({\xi}/2,{\tau}/2),
\end{equation}
when $f$, $g\in \Cal S(\RR^n)$.
\end{lemma}

\begin{proof}
Let $\phi ({\xi},t)$ be the Fourier transform of $A(f,g)(x,t)$ in
the variable $x$. 
Then
$$
\begin{gathered}
{\phi}({\xi},t)= (2\pi)^{-n} \iint k_0(y, t) \widehat f({\eta})\widehat
g({\xi}-{\eta})\ee^{-\ii \scalar{2 \eta -\xi}{y}}\dd y \dd {\eta}\\
=  (2\pi)^{-n} 2^{-n} \iint k_0(y, t) \widehat f(\frac{\xi+\eta}{2}) \widehat
g(\frac{\xi-\eta}{2})\ee^{-\ii \scalar{ \eta}{y}}\dd y \, \dd {\eta}\\
=(2\pi)^{-n} 2^{-n} \int \frac{\sin(|\eta|t)}{|\eta|} \widehat f(\frac{\xi+\eta}{2}) \widehat
g(\frac{\xi-\eta}{2}) \,\dd \eta.\\
\end{gathered}
$$
It follows that
$$
\begin{gathered}
\widehat A(f,g)({\xi},{\tau}) = 
(2\pi)^{-n} 2^{-n} \iint\limits_{\RR^n \times \RR}   
\ee^{-\ii \tau t} \frac{ \ee^{\ii t |\eta|}-\ee^{-\ii  t |\eta|}}{2\ii |\eta|}
\widehat f(\frac{\xi+\eta}{2}) \widehat
g(\frac{\xi-\eta}{2})\dd t  \,\dd {\eta}\\
=(2\pi)^{-(n-1)} \ii^{-1} 2^{-(n+1)}|{\tau}|^{-1} \int\limits_{\RR^n}   
(\delta(|\eta|-\tau)-\delta(|\eta|+\tau))
\widehat f(\frac{\xi+\eta}{2}) \widehat
g(\frac{\xi-\eta}{2})  \,\dd {\eta}\\
=(2\pi)^{-(n-1)} \ii^{-1} 2^{-(n+1)} \tau^{n-2} \int\limits_{\bS^{n-1}}   
\widehat f(\frac{\xi+\tau \omega}{2}) \widehat
g(\frac{\xi-\tau \omega }{2})  \,\dd {\omega}.
\end{gathered}
$$
This combined with \eqref{def:B} gives the lemma.
\end{proof}

\begin{lemma}\label{lemma:k0}
We have
\begin{equation}\label{k0:1}
k_0(x, t) =\partial_t^m \kappa_0(x, t),
\end{equation}
where the smooth mapping 
$ \RR \ni t\to \kappa_0(\cdot, t)\in {\Cal D}'(\RR^n)$ is given by
$$ \scalar{\kappa_0(\cdot, t)}{\varphi}=
\pi (2\pi)^{-(n+1)/2} t^{m+1} \int\limits_{\bS^{n-1}}
\varphi(t\omega)\dd \omega +  \int\limits_t^\infty  p(t/r) r^m \int\limits_{\bS^{n-1}}
 \varphi(r\omega) \dd \omega \dd r$$
for every  $\varphi \in C_0^\infty(\RR^n)$.
Here  
$$p(s) =\frac{1}{m! (4\pi)^{m+1}}
(-\frac{\dd }{\dd s})^{m+1} (1-s^2)^m.
$$
\end{lemma}  

\begin{proof}
We notice that $p(s)$ is a polynomial of degree $m-1$ which is  odd
(even) if $m $ is  even (odd). 
The polynomial $r^m p(t/r)$  is
therefore  odd in $r$ and  odd (even) in $t$ if $m$ is even (odd).  
Set
$\widetilde \varphi (t) = \int \limits_{{\bS}^{n-1}} \varphi
(t{\omega}) \dd{\omega}$ when $t \in \RR$. 
This a smooth and even
function of $t$. 
If $m$ is even then
$$
\begin{gathered}
\int\limits _t^\infty p(t/r) r^m \widetilde \varphi (r) \dd r  
=\int\limits _{-t}^\infty p(t/r) r^m \widetilde \varphi (r) \dd r
=- \int\limits _{-t}^\infty p(-t/r) r^m \widetilde \varphi (r) \dd r
\end{gathered}
$$
which shows that the left-hand side is an odd function of $t$. 
If $m$
is odd similar arguments show that the left-hand side is even in $t$. 
Hence, if  we  define ${\kappa}_0$  as in the lemma it follows that
${\kappa}_0(\cdot, t)$ is a smooth distribution valued function of $t$ which is
odd (even) if $m$ is  even (odd).

Define 
$$
U_0(x,t) = 
\int \limits _{\bS^{n-1}}\delta \upp{m+1}(\scalar x {\omega}
-t ) \dd{\omega}.
$$
It follows from equation (5.4) in \cite{M:contemp} that 
$$
k_0(x, t) = \p _t ^m {\pi}(2{\pi}) ^{-n}U_0(x,t).
$$
Here $U_0(\cdot, t)$ is a smooth distribution valued function of $t$
with the same parity as ${\kappa}_0(\cdot ,t)$. 
The lemma follows
therefore if we prove that
$$
{\kappa}_0(x, t)={\pi} (2{\pi})^{-n}U_0(x,t)
$$
when $t >0$.

We may write
$$
\begin{aligned}
U_0(x, t)  & = (-\partial_t)^{m+2} \int\limits_{\bS^{n-1}} 
Y_+(x\omega -t) \dd\omega\\
& = c_{n-1} (-\partial_t)^{m+2}  \int\limits_{-1}^1 Y_+(|x|s-t) (1-s^2)^m \dd s\\
& = c_{n-1} (-\partial_t)^{m+1}  \int\limits_{-1}^1 \delta(|x|s-t) (1-s^2)^m \dd s.
\end{aligned}
$$ 
where
$
c _{n-1} = 2 {\pi}^{m+1}/m!$ is the area of the $(n-2)$-dimensional
unit sphere.  
In $\{t>0 \}$ we have
$$ 
\begin{gathered}
 \scalar{U_0(\cdot, t)}{\varphi} = c_{n-1}(-\partial_t)^{m+1}\int\limits_0^\infty \int\limits_{-1}^{1} 
\delta(sr-t) r^{n-1} (1-s^2)^m \tilde\varphi(r)\dd s\dd r\\
= c_{n-1} (-\partial_t)^{m+1} \int\limits_0^1 (t/s)^{n-1} (1-s^2)^m s^{-1} \tilde\varphi (t/s) \dd s
\\ = c_{n-1} (-\partial_t)^{m+1} \int \limits_t^\infty r^{n-2} (1-t^2/r^2)^m 
\tilde\varphi (r) \dd r.
\end{gathered}
$$
Set $q(s) = c_{n-1}(-\dd/\dd s)^{m+1}(1-s^2)^m ={\pi}^{-1}(2{\pi})^n p(s)$. 
A simple computation then gives
$$
\begin{gathered}
\scalar{U_0(\cdot ,t)}{\varphi} = c_{n-1} m! 2^m t^{m+1} \widetilde
\varphi (t) + \int\limits _t ^\infty q(t/r)r^m \widetilde \varphi (r) \dd r
\\
={\pi}^{-1}(2{\pi})^n
\scalar{{\kappa}_0(\cdot t)}\varphi.
\end{gathered}
$$
This finished the proof of the lemma.
\end{proof}

\begin{corollary}\label{formA}
With the notation in the previous lemma, we have
$$ A(f, g)(x, t) = \partial_t^m \left ( \pi (2\pi)^{-(n+1)/2} S(f, g)(x, t) 
+\int\limits_t^\infty p(t/r) r^{-1} S(f, g)(x, r) \dd r\right)
$$ for every $f$, $g\in C_0^\infty(\RR^n)$.
\end{corollary}

It has become clear that, in order to get continuity properties of $A$, 
we need to study the bilinear operator $S$. 
We start with an elementary lemma, where $\opn{meas}(\cdot)$ denotes 
the surface measure on $\bS^{n-1}$.
\begin{lemma}\label{lemma:1.1}
There is a constant $C$ such that
\begin{equation}\label{measure:1}
\opn{meas}\left( \{{\omega}\in \bS^{n-1};\, r/2<|x-t\omega|<2r, \,
|x+t\omega|<s  \}\right)
\leq C \Big (\frac{s}{r}\Big )^{n-1} 
\end{equation}
when $r$, $s>0$, $ x\in \RR ^n$, $t\in \RR$.
\end{lemma}

\begin{proof}
It is enough to prove the lemma for $s< r/4$.
Denote
$$
M(x,t;r,s ) = \{{\omega} \in \bS^{n-1};\, r/2<|x-t\omega| <2r,\
|x+t\omega|< s \}.
$$
Since
$$ 
\opn{meas}(M(x,-t;r,s ))=\opn{meas}(M(x,t;r,s )) 
$$
we may assume $t\ge 0$.

If $\omega \in M(x,t;r, s)$ we must have $\scalar{x}{\omega}\le 0$.
It follows that $r/2 < |x|+t< 2\sqrt 2 r$ when 
$   M(x,t;r, s)\ne  \emptyset$. 
Also $||x|-t|< s< r/4$, hence $|x|$, $t$ and $r$ are of the same order
of magnitude. 
Using the fact that the push-forward of the  measure $\dd{\omega}$ on
$\bS^{n-1}$ under the mapping ${\omega}\mapsto {\tau} = \scalar x
{\omega}/|x| \in [-1,1]$   
is a multiple of the measure $(1-{\tau}^2)^m \dd{\tau}$  we easily
see that 
$$
\opn{meas}(M(x,t;r, s)) \leq C \int\limits_{N(x,t,s)}(1-\tau)^m\dd \tau,
$$
where
$$
N(x,t,s) =\{\tau \in (0,1);\, |x|^2+t^2 -2|x|t\tau < s^2 \}.
$$
Since 
$$ 
\int\limits_{N(x,t,s)}(1-\tau)^m\dd \tau \le 
\int\limits_0^\frac{s^2-(|x|-t)^2}{2|x|t} \tau^m \dd \tau \le C \left(\frac{s^2-(|x|-t)^2}{2|x|t}\right)^{m+1},
$$ 
we  have proved that
$$
\opn{meas}(M(x,t;r, s))  \leq  C(s^2/(|x|t))^{\frac{n-1}{2}}.
$$
Recalling that $|x|$ and $t$ are of the same order of magnitude as $r$,
we see that \eqref{measure:1} holds.
\end{proof}

\begin{lemma}\label{lemma:op2}
Assume $r$, $ s>0$, $\phi$, $\psi\in C_0^\infty(\RR^n)$, $\phi$ is supported in the set
where $r/2< |x|< 2r$, $\psi$ is supported in the set where $|x|< s$, and $a\in \RR$.
Then there is a constant $C=C(a)$, independent of $r$, $s$, $\phi$ and $\psi$,
such that
\begin{align}
\normrum{S(\phi , \psi )}{(a, 0)} & \le C (s/r)^{m+1} \max(\hake{r}^a, \hake{r+s}^a) \norm{\phi }\,
\norm{\psi }. \label{op:2}
\end{align}
\end{lemma}

\begin{proof}
It follows  from Lemma~\ref{lemma:1.1} and Cauchy's inequality
applied to the integration with respect to  ${\omega}$ that there is a constant
$C$ such that
\begin{equation}\label{main:101}
| S(\phi ,\psi )(x, t)|^2 \leq 
C t^{n-1}  (s/r)^{2(m+1)} 
\int |\phi (x+ t {\omega})|^2|\psi (x- t {\omega})|^2 \dd {\omega}.
\end{equation}
Since $2(|x|^2 +t^2) =|x+t\omega|^2 + |x-t\omega|^2$  when $\omega\in S^{n-1}$, one has
\begin{equation}\label{main:102}
\begin{gathered}
(1+|x|^2+t^2)^a | S(\phi ,\psi )(x,t)|^2 \\
\le 
C t^{n-1} \left(\frac{s}{r}\right)^{2(m+1)} 
\int 
(1+ |x+t\omega|^2 + |x-t\omega|^2)^a 
|\phi (x+t {\omega})|^2|\psi (x- t{\omega})|^2 \dd {\omega}\\
\le C_1 t^{n-1} \left(\frac{s}{r}\right)^{2(m+1)} \max(\hake{r}^{2a}, \hake{r+s}^{2a})
\int |\phi (x+ t {\omega})|^2| \psi(x- t {\omega})|^2 \dd {\omega}.
\end{gathered}
\end{equation}
An integration  with respect to $x$ and $t$ in \eqref{main:102}
gives $\eqref{op:2}$.
\end{proof}

\begin{lemma}\label{mainS}
Let  $a'$, $a''$, $a\in \RR$ satisfy
\begin{equation}\label{cond:a}
a< m+1 +\min(a', a''), \quad a\le a'+a''.
\end{equation}
Then $S$ is continuous from $H_{(a', 0)} \times H_{(a'', 0)}$ to 
$ H_{(a, 0)}$.
\end{lemma}

\begin{proof}
Choose $\chi\in C_0^\infty (\RR^n) $ a smooth decreasing function of $|x|$
such that $\chi(x)=1$  when $|x|< 1$, 
$\chi(x) =0$ when $|x|>2$ and $0\le \chi\le 1$.
Set
$$
\chi_j(x) = \chi(2^{-j} x) - \chi(2^{1-j}x), \quad  j\ge 1, \quad \chi_0(x) =\chi(x)$$
when $x\in \RR^n$. 
Then 
$f =\sum\limits_{0}^\infty \chi_jf$ with convergence in  ${\Cal S}(\RR^n)$
when $f$ is in that space. 
In addition, when $\rho \in \RR$,  there is $C= C(\rho) \ge 0$ such that
\begin{equation}\label{mainS:2}
C^{-1}\sum\limits_0^\infty 2^{2\rho j} \norm{\chi_j f}^2 
\le  \norm{f}^2_{(\rho , 0)} \leq  C \sum\limits_0^\infty 2^{2\rho j}
\norm{\chi_j f}^2.  
\end{equation}

Consider $f$, $g \in {\Cal S}(\RR^n)$ and denote $s_j= 2^{a'j} \norm{\chi_jf}$, $\sigma_k 
=2^{a'' k} \norm{\chi_k g}$. 
These are $\ell^2(\NN)$-sequences with $\ell^2$ norm bounded from 
above by  a constant $C$ times $\normrum{f}{(a', 0)}$ and 
$\normrum{g}{(a'', 0)}$, respectively.
Set $\ep = m+1+\min (a', a'')-a$. 
Then  $\ep >0$ and we shall show that
there is a constant $C>0$, which depends on 
$a$, $a'$, $a''$ only, such that
\begin{equation}\label{mainS:3}
\normrum{S(\chi_j f, \chi_k g)}{(a, 0)} \le C 2^{-\ep|j-k|}
s_j \sigma_k.
\end{equation}
Hence
$$ 
\begin{gathered}
\normrum{S(f, g)}{(a, 0)} \le \sum\limits_{j, k\ge 0} 
\normrum{S(\chi_j f, \chi_k g)}{(a, 0)}\le 
 C_1 \sum\limits_{j, k\ge 0}  2^{-\ep|j-k|}s_j \sigma_k\\
\le
C_2 (\sum\limits_{j\ge 0} s_j^2)^{1/2}(\sum\limits_{j\ge 0} \sigma_k^2)^{1/2}
\le C \normrum{f}{(a', 0)} \normrum{g}{(a'', 0)}.
\end{gathered}
$$
This would prove the statement.

It remains to prove \eqref{mainS:3}. 
Since $S$ is symmetric, and since the
condition \eqref{cond:a} is symmetric in $(a',a'')$, 
it suffices to prove
\eqref{mainS:3} when $j \geq k$.  
The previous lemma shows that
\begin{equation}\label{mainS:4}
|S(\chi _jf, \chi _kg)| \leq C 2^{-{\rho}_{jk}}s_j{\sigma}_k,
\end{equation} 
where
$$
\begin{gathered}
{\rho}_{jk} = (j-k)(m+1)-aj+a'j+a''k\\
=(j-k)(m+1+a'-a) +(a'+a''-a)k \\
\geq (j-k)(m+1+\min (a', a'')-a)= (j-k)\ep.
\end{gathered}
$$
This proves \eqref{mainS:3}.
\end{proof}

\begin{lemma}\label{lemma:3.9}
Define
$$
T(f,g)(x,t)= \int\limits_t ^\infty p(t/r)r^{-1} S(f,g)(x,r) \dd r,
$$
when $f, g \in C_0^\infty (\RRn)$. 
Then
$$
\normrum{T(f,g)}{(a,0)}\leq 2 \max _{|s|\leq 1 }|p(s)|\cdot 
\normrum{S(f,g)}{(a,0)}
$$
when $a \geq 0$.
\end{lemma}

\begin{proof}
We recall that
$$
p(t/r)r^{-1} S(f,g)(x,r)= p(t/r)r^m r^{-(m+1)}S(f,g)(x,r) 
$$
is an odd function of $r$. 
Hence
$$
\begin{gathered}
|T(f,g)(x,t)| \leq \Big\vert \int\limits _{|t|} ^\infty p(t/r)r^{-1}S(f,g)(x,r)\dd r
\Big\vert 
\leq C \int\limits _{|t|}^\infty r^{-1} |S(f,g)(x,r)| \dd r
\end{gathered}
$$
where $C= \max _{|s|\leq 1}|p(s)|$. 
When $a \geq 0$ we get
$$
\begin{gathered}
(1+|x|^2+t^2)^{a/2}|T(f,g)(x,t)| 
\leq C \int\limits _{|t|}^\infty r^{-1} (1+|x|^2+r^2)^{a/2} |S(f,g)(x,r)| \dd r
\end{gathered}
$$
The lemma follows therefore if we  notice that
$$
\int\limits_0 ^\infty H^2(t) \dd t \leq 4\int \limits_0^\infty h^2(t)\dd t
$$
when $H(t) = \int\limits _t ^\infty t^{-1} h(t) \dd t$ and $0 \leq h \in
C_0(\RR)$. 
In fact, if $\widetilde h(s)=e^{s/2}h(e^s)$ and 
$\widetilde H(s)=e^{s/2}H(e^s)$, then 
$$
\int\limits_0 ^\infty h^2(t) \dd t = \normrum {\widetilde h}{L^2(\RR )}^2, \ \ 
\int\limits _0 ^\infty H^2(t) \dd t = \normrum {\widetilde H}{L^2(\RR )}^2,
  $$
and $\widetilde H= {\gamma} \ast \widetilde h$, where ${\gamma}(s) =
(1-Y_+(s))e^{s/2}$  has  $L^1$ norm equal to $2$.
\end{proof}

\begin{proposition}\label{est:A}
\rm{(i)}  When $a'$, $a''$, $a\in \RR$ satisfy
$$ 0\le a < m+1 +\min(a', a''), \quad a\le a'+a'',$$ 
then $A$ extends to a continuous bilinear operator from $H_{(a', 0)}\times H_{(a'', 0)}$
to $H_{(a, -m)}$.

\rm{(ii)}  When $b'$, $b''$, $b\in \RR$ satisfy
$$ b < m+1 +\min(b', b''), \quad b\le b'+b'',$$ 
then $A$ extends to a continuous bilinear operator from 
$H_{(0, b')}\times H_{(0, b'')}$ to $H_{(0, b-m)}$.
\end{proposition}

\begin{proof}
A combination of Corollary~\ref{formA}, Lemma~\ref{mainS} and Lemma~\ref{lemma:3.9} gives (i), and
(ii) follows from Lemma~\ref{lemma:fourier} and Lemma~\ref{mainS}.
\end{proof}

\begin{proposition}\label{main:A}
Let $(a', b', a'', b'', a, b )\in \RR^6$ satisfy
\begin{equation}\label{condA}
\begin{gathered}
0\le a< m+1 +\min(a', a''),\; a\le a'+a'',\; b< m+1 +\min(b', b''),\; b\le b'+b'',\\
a+b< m+1 +\min(a', a'') + \min(b',b'').
\end{gathered}
\end{equation}
Then $A$ is continuous from $H_{(a', b')}\times H_{(a'', b'')}$ to
$H_{(a, b-m)}$.
\end{proposition}
 
\begin{proof}
If $(a', b', a'', b'', a, b )\in \RR^6$ satisfy \eqref{condA}, it is easy to
see that there is an $\theta\in (0, 1)$ such that
$$a< \theta(m+1) +\min(a', a''), \; b<(1-\theta)(m+1) +\min(b', b'').
$$
This shows that $a'/\theta$, $a''/\theta$, $a/\theta$, respectively
$b'/(1-\theta)$, $b''/(1-\theta)$, $b/(1-\theta)$
satisfy the conditions in Proposition~\ref{est:A}, hence 
$$
\begin{aligned}
(a'/\theta, 0,  a''/\theta, 0,  a/\theta, -m) & \in I(A),\\
(0, b'/(1-\theta), 0, b''/(1-\theta), 0, -m+ b/(1-\theta))& \in I(A).
\end{aligned}
$$
The proposition then follows by Theorem~\ref{interp:thm}.
\end{proof}

\begin{proof}[Proof of Theorem~\ref{mainthm}]
Theorem~\ref{mainthm} follows now directly from
Proposition~\ref{main:A} and Corollary~\ref{cor:K}.
\end{proof}

\end{document}